\documentclass[11pt]{amsart}

\usepackage[dvips,final]{graphicx}
\usepackage[dvips]{geometry}
\usepackage{color}
\usepackage{pstricks}
\usepackage{amssymb,amscd,latexsym,epsfig,xypic,hyperref}
\usepackage[mathscr]{euscript}

\DeclareFontFamily{OT1}{rsfs}{}
\DeclareFontShape{OT1}{rsfs}{n}{it}{<-> rsfs10}{}
\DeclareMathAlphabet{\curly}{OT1}{rsfs}{n}{it}

\DeclareFontFamily{OT1}{arrow}{\hyphenchar\font45 }
\DeclareFontShape{OT1}{arrow}{m}{n}{<7><8><10><12><13.82><16.59><19.907><23.89><28.66><34.4><41.28>cmr8}{}
\DeclareSymbolFont{arrowsym}{OT1}{arrow}{m}{n}
\DeclareMathSymbol{\arroweq}{\mathrel}{arrowsym}{"3D}

\renewcommand\;{\hspace{.6pt}}
\renewcommand\P[1]{{\mathbb P}^{\;#1}}
\newcommand\PP{\mathbb P}
\newcommand\LL{\mathbb L}
\newcommand\I{\curly I}
\newcommand\C{\mathbb C}

\newcommand\Z{\mathbb Z}

\newcommand\cO{\mathcal O}
\newcommand\cA{\mathcal A}

\newcommand\cC{\mathcal C}

\renewcommand\cH{\mathcal H}

\newcommand{\rt}[1]{\stackrel{#1\,}{\rightarrow}}
\newcommand{\Rt}[1]{\stackrel{#1\,}{\longrightarrow}}
\newcommand{\RT}[2]{\xymatrix@C=#1pt{\ar[r]^{#2}&}}
\newcommand\To{\longrightarrow}
\newcommand\oT{\longleftarrow}

\newcommand\into{\hookrightarrow}
\newcommand\INTO{\ \ar@{^(->}[r]<-.2ex>}
\newcommand{\Into}{\ensuremath{\lhook\joinrel\relbar\joinrel\rightarrow}}

\renewcommand\_{^{}_}

\newcommand\Langle{\big\langle}
\newcommand\Rangle{\big\rangle}
\newfont{\bigtimesfont}{cmsy10 scaled \magstep5}
\newcommand{\bigtimes}{\mathop{\lower0.9ex\hbox{\bigtimesfont\symbol2}}}

\newcommand\Gr{\operatorname{Gr}}
\newcommand\Pf{\operatorname{Pf}}
\newcommand\Sym{\operatorname{Sym}}

\newcommand\rk{\operatorname{rank}}

\newcommand\id{\operatorname{id}}

\newcommand\Hom{\operatorname{Hom}}
\renewcommand\hom{\curly H\!om}

\newcommand\Ext{\operatorname{Ext}}

\newcommand\Spec{\operatorname{Spec}\,}
\newcommand\Hilb{\operatorname{Hilb}}
\newcommand\Cone{\operatorname{Cone}}
\newcommand\Bl{\operatorname{Bl}}

\renewcommand\;{\hspace{.6pt}}

\newcommand\beq[1]{\begin{equation}\label{#1}}
\newcommand\eeq{\end{equation}}
\newcommand\beqa{\begin{eqnarray*}}
\newcommand\eeqa{\end{eqnarray*}}

\DeclareRobustCommand{\SkipTocEntry}[3]{}
\makeatletter
\newcommand\@dotsep{4.5}
\def\@tocline#1#2#3#4#5#6#7{\relax
  \ifnum #1>\c@tocdepth 
  \else
    \par \addpenalty\@secpenalty\addvspace{#2}%
    \begingroup \hyphenpenalty\@M
    \@ifempty{#4}{%
      \@tempdima\csname r@tocindent\number#1\endcsname\relax
    }{%
      \@tempdima#4\relax
    }%
    \parindent\z@ \leftskip#3\relax \advance\leftskip\@tempdima\relax
    \rightskip\@pnumwidth plus1em \parfillskip-\@pnumwidth
    #5\leavevmode #6\relax
    \leaders\hbox{$\m@th
      \mkern \@dotsep mu\hbox{.}\mkern \@dotsep mu$}\hfill
    \hbox to\@pnumwidth{\@tocpagenum{#7}}\par
    \nobreak
    \endgroup
  \fi}
\makeatother

\makeatletter \@addtoreset{equation}{section} \makeatother

\newtheorem{defn}[equation]{Definition}
\newtheorem{thm}[equation]{Theorem}
\newtheorem{lem}[equation]{Lemma}

\newtheorem{prop}[equation]{Proposition}
\theoremstyle{definition}\newtheorem{rmk}[equation]{Remark}

\title{Notes on Homological Projective Duality}
\author{Richard P. Thomas}

\begin{document}

\maketitle
\renewcommand\contentsname{\vspace{-9mm}}
\tableofcontents \vspace{-6mm}


\section{Introduction}
Kuznetsov's Homological Projective Duality \cite{HPD} is a beautiful way to relate the derived categories of coherent sheaves on different varieties.

In these notes we begin with Beilinson's and Orlov's theorems on the derived categories of projective bundles and blow ups, and show how these lead naturally and easily to what in these notes we call ``HPD I" (Kuznetsov calls it ``HPD for the stupid Lefschetz decomposition"). It is a simple manifestation of the natural geometric correspondence between a projective variety $$X\To\PP(V^*)$$ and its projective dual family of hyperplane sections
\beq{Hpdual}
\cH\To\PP(V).
\eeq

Then we describe Kuznetsov's theory of Lefschetz decompositions to show how to cut down $D(\cH)$ in HPD I to its ``interesting part" $\cC_\cH\subset D(\cH)$ in the more economical HPD II. If the reader takes some lengthy, tedious cohomology and mutation calculations on trust, she can easily and quickly understand HPD II as a simple rearrangement of Lego blocks in HPD I, with the whole abstract theory described in a few pages. Remark \ref{HPDIII} gives perhaps the best way to summarise HPD II.

I personally found that splitting HPD into these two steps, and understanding each separately, makes the theory much easier to follow.\footnote{In fact, although I was the MathSciNet reviewer of \cite{HPD}, I only felt I understood it properly much later, when I saw it from the point of view described in these notes.}  In some sense it can be thought of as simply reading the original paper \cite{HPD} backwards. But it gives nothing new, apart from sidestepping Kuznetsov's requirement that the HP dual should be \emph{geometric} --- which was surely never a particularly important prerequisite in his theory anyway. Indeed it gives less, since we restrict to the case of rectangular Lefschetz collections, and smooth baseloci, to simplify the exposition. The extension to non-rectangular Lefschetz collections is handled in \cite{JLX}. The way to handle the singular case is to repeat the working in generic (linear) families with smooth total space, as in \cite{HPD}. That way one also gets a relative version of HPD.

Of course the power and interest of HPD is that geometric examples actually exist. The category $\cC_\cH\subset D(\cH)$  that we cut down to in HPD\nolinebreak\ II is only interesting because of Kuznetsov's great insight that it is often geometric --- i.e. it is equivalent to $D(Y)$ for some variety $Y\to\PP(V)$. Then the HP dual of $X\to\PP(V^*)$ is another variety $Y\to\PP(V)$ rather than just some abstract $D(\PP(V))$-linear triangulated category $\cC_\cH$. In this case we get a true duality, with $X$ also the HP dual of $Y$. We survey some of these geometric examples in Section \ref{egs}. 

(In these notes we are deliberately vague about what we mean by a category being geometric. The best situation is when it is $D(Y)$, but often it will only be this locally; there may be a Brauer class to twist by, and singularities which admit a noncommutative or categorically crepant resolution.)

Since the derived category $D(H)$ --- and its ``interesting part" $\cC_H\subset D(H)$ --- can detect when a projective variety $H$ is smooth, we can recover the classical projective dual $\check X$ of $X$ (i.e. the discriminant locus of $\cH\to\PP(V)$ \eqref{Hpdual}) from HPD. Since $\check X$ is the locus of hyperplanes $H$ of $X$ which are singular, it is the locus of $H\in\PP(V)$ for which $\cC_H$ fails to be smooth and proper. In the geometric case this is the same as the fibre $Y_H=Y\times_{\PP(V)}\{H\}$ being singular, so $\check X$ is also the discriminant locus of $Y\to\PP(V)$.

\smallskip\noindent\textbf{Acknowledgements.} The results here are all due to Sasha Kuznetsov. I thank him for his patience and good humour in answering all my questions, and for numerous helpful comments on the manuscript. Thanks also to Nick Addington, Marcello Bernardara, David Favero, J\o rgen Rennemo and Ed Segal for HPD discussions over several years. I am grateful to Qingyuan Jiang, Conan Leung and Ying Xie for suggesting the proof of generation in Theorem \ref{hpdthm}, and two conscientious referees for corrections.

\section{Projective bundles and blow ups}

Throughout $D(X)$ denotes the bounded derived category of coherent sheaves on a smooth complex projective variety $X$. All functors (restriction, pullback, pushforward, tensor product) are derived unless otherwise stated.

\subsection{Beilinson}\label{bile} We begin with Beilinson's theorem
\begin{eqnarray} \label{Beil}
D(\P{n-1}) &=& \Langle\cO,\cO(1),\ldots,\cO(n-1)\Rangle \\ \nonumber
&=& \Langle D(\mathrm{pt}),D(\mathrm{pt}),\ldots,D(\mathrm{pt})\Rangle.
\end{eqnarray}
What we mean by this notation (which we shall use repeatedly) is that
\begin{itemize}
\item The sheaves $\cO,\cO(1),\ldots,\cO(n-1)$ \emph{generate} $D(\P{n-1})$,
\item They are \emph{semi-orthogonal}: there are no RHoms from right to left, i.e. $\Ext^*(\cO(i),\cO(j))=0$ for $i>j$,
\item The functor $D(\mathrm{pt})=D(\Spec\C)\to D(\P{n-1})$ which takes $\C$ to $\cO(i)$ is an embedding --- i.e. it is full and faithful. Equivalently $\cO(i)$ is \emph{exceptional}: $R\Hom(\cO(i),\cO(i))=\C\cdot\id$.
\end{itemize}
The second two conditions are simple cohomology computations. We sketch how to recover the first condition from them.\footnote{As Nick Addington showed me.} Since $\cO,\cO(1),\ldots,\cO(n-1)$ are exceptional and semi-orthogonal, the Gram-Schmidt process\footnote{For instance, to project $E$ to $\langle\cO\rangle^\perp$ we replace it by $\Cone\big(\cO\otimes R\Hom(\cO,E)\to E\big)$.} shows that any object of $D(\P{n-1})$ is an extension of a piece in their span and a piece in their right orthogonal. Thus we want to show their orthogonal is zero.

Any point $p\in\P{n-1}$ is cut out by a section of $\cO(1)^{\oplus(n-1)}$, so its structure sheaf $\cO_p$ admits a Koszul resolution by sheaves in $\Langle\cO,\cO(1),\ldots,\cO(n-1)\Rangle$. Thus any object $F$ in the right orthogonal is also orthogonal to $\cO_p$, and
\beq{naky}
F|_p\,\cong\,R\hom(\cO_p,F)[n-1]\,=\,0
\eeq
for any $p\in\P{n-1}$. Therefore $F=0$ by a standard argument. (Replacing $F$ by a quasi-isomorphic finite complex $F^\bullet$ of locally free sheaves, by \eqref{naky} its restriction to $p$ is an exact complex of vector spaces. In particular the final map is onto so, by the Nakayama lemma, the final map in $F^\bullet$ is also onto in a neighbourhood of $p$. Thus we can locally trim to a shorter quasi-isomorphic complex of locally free sheaves by taking its kernel. Inductively we find that $F$ is quasi-isomorphic to zero in a neighbourhood of $p$ for every $p$.)

\smallskip\noindent\textbf{Digression.} We also describe Beilinson's original proof, which is beautifully geometric and actually proves more --- giving an explicit presentation \eqref{explic} of any $F\in D(\PP^{n-1})$ in terms of the exceptional collection.

We set $\PP:=\PP^{n-1}=\PP(V)$ for some vector space $V\cong\C^n$. Then on $\PP\times\PP$ there is a canonical section of $T_{\PP}(-1)\boxtimes\cO_{\PP}(1)$ corresponding to $\id_V\in V\otimes V^*\cong H^0(T_{\PP}(-1))\otimes H^0(\cO_{\PP}(1))$. Up to scale, at the point $(p_1,p_2)\in\PP\times\PP$, it gives the tangent vector at $p_1$ of the unique line from $p_1$ to $p_2$. But it vanishes precisely when the $p_i$ coincide, i.e. on the diagonal $\Delta$. Thus we get a Koszul resolution for $\cO_\Delta$ (in fact the one used above for $\cO_p$, put into the universal family as $p$ moves over $\PP$),
$$
0\To\Omega^{n-1}_{\PP}(n-1)\boxtimes\cO_{\PP}(1-n)\To\ldots\To\Omega_{\PP}(1)\boxtimes\cO_{\PP}(-1)\To\cO\To\cO_\Delta\To0.
$$
Therefore the complex $\Omega^\bullet_{\PP}(\bullet)\boxtimes\cO_{\PP}(-\bullet)$ is (quasi-isomorphic to) the  Fourier-Mukai kernel $\cO_\Delta$ for the identity functor. Applying it to any $F\in D(\PP)$, we see $F$ is quasi-isomorphic to the complex
\beq{explic}
F_{n-1}\otimes\_\C\cO_{\PP}(1-n)\To\ldots\To
F_1\otimes\_\C\cO_{\PP}(-1)\To F_0\otimes\_\C\cO_{\PP\;},
\eeq
where $F_i:=R\Hom\big(\Lambda^iT_{\PP}, F(i)\big)$.
Thus $F\in\Langle\cO,\cO(-1),\ldots,\cO(1-n)\Rangle$.

\subsection{Orlov I} This is a family version of Beilinson's result. Suppose $X=\PP(E)\Rt\pi B$ is the $\P{n-1}$-bundle of a rank $n$ vector bundle $E$ over a smooth base $B$. It has a tautological bundle $\cO(-1)\into\pi^*E$. Orlov \cite{Or} showed that
$$
D(X)\,=\,\Langle D(B),\,D(B)(1),\,\ldots,\,D(B)(n-1)\Rangle.
$$
Here the $i$th copy of $D(B)$ is embedded via $\pi^*(\ \cdot\ )\otimes\cO(i)$.
Again easy cohomological computations show these are semi-orthogonal embeddings. For generation, see the slightly more general proof of HPD I \eqref{hypthm} below.

\smallskip\noindent\textbf{Digression.} $\P{n-1}$-bundles which are \emph{not} the projectivisation of a vector bundle can also be handled \cite{Be} by using the \emph{twisted} derived categories $D(B,\alpha^i)$, where $\alpha\in H^2_{\mathrm{\mathaccent19 et}}(\cO_B^*)$ is the Brauer class of the bundle.

\subsection{Orlov II} There is a similar result for the derived category of a blow up $\Bl_Z X$, where $Z\subset X$ are both smooth. We use the notation
$$
\xymatrix{
E \ar[d]_p\INTO^-j& \Bl_ZX \ar[d]^\pi \\
Z \INTO^i& X}
$$
for the blow up $\pi$ and its exceptional divisor $E$. Including $D(X)$ into $D(\Bl_ZX)$ by  $\pi^*$, and $D(Z)$ by $j_*p^*$, Orlov \cite{Or} proved
$$
D(\Bl_ZX)\,=\,\Langle D(X),\,D(Z),\,D(Z)(-E),\,\ldots,\,D(Z)(-(n-2)E)\Rangle,
$$
where $n$ is the codimension of $Z\subset X$. Since $p\colon E\to Z$ is a $\P{n-1}$-bundle whose $\cO(1)$ line bundle is $\cO(-E)$, the
analogy to Orlov I --- with one copy of $D(Z)$ replaced by $D(X)$ --- is clear. That $\pi^*$ and $j_*p^*$ give embeddings is an easy cohomology computation, as is semi-orthogonality.
For generation see the proof of HPD I \eqref{hypthm} below.  \medskip

We now turn to homological projective duality. We emphasise that
a large part of it can be seen as an Orlov-type theorem for the derived category of a certain fibration over $X$. This fibration is generically a projective bundle with fibres $\P{\ell-2}$, but over a codimension-$\ell$ subvariety $X_{L^\perp\!}\subset X$ these jump to $\P{\ell-1}$ (so the case $\ell=2$ is the blow up above).

\section{Homological projective duality I}
\subsection{The universal hyperplane section}\label{qu}
Fix $(X,\cO_X(1))$, a variety with a semi-ample line bundle, and a basepoint-free linear system
\beq{inc}
V\subseteq H^0(\cO_X(1)).
\eeq
We get a map
\beq{f}
f\colon X\To\PP(V^*)
\eeq
whose image is not contained in a hyperplane. Conversely, given a map \eqref{f} with image not contained in a hyperplane, we recover a basepoint-free linear subsystem \eqref{inc} by pullback $f^*$ on sections of $\cO(1)$.

The natural variety over the \emph{dual} projective space $\PP(V)$ is the universal family of hyperplanes $\cH\to\PP(V)$, where
\beq{cH}
\cH:=\big\{(x,s):s(x)=0\big\}\subset X\times\PP(V)
\eeq
is the obvious incidence hyperplane. Here and below we identify $s\in V$ with a section $s\in H^0(\cO_X(1))$ via \eqref{inc}. The isomorphism
\beq{tort}
H^0\big(\cO_{X\times\PP(V)}(1,1)\big)\,\cong\,H^0(\cO_X(1))\otimes V^*
\eeq
defines a tautological section corresponding to $f^*$ \eqref{inc}. This cuts out the divisor $\cH$ \eqref{cH}. The discriminant locus of the projection $\cH\to\PP(V)$ is the classical projective dual of $X$.

Given a linear subspace $L\subseteq V$, let $\ell$ denote its dimension and $L^\perp\subseteq V^*$ its annihilator. We set
\beqa
X_{L^\perp} \!\!&:=&\!\! X\times_{\PP(V^*)}\PP(L^\perp), \\
\cH_L \!\!&:=&\!\! \cH\times_{\PP(V)}\PP(L).
\eeqa
Notice that $X_{L^\perp}$ is nothing but the baselocus of the linear system $$L\subseteq H^0(\cO_X(1))$$ and \emph{therefore is contained in every fibre of $\cH_L\to\PP(L)$}, giving a diagram
\beq{peche}
\xymatrix{
X_{L^\perp}\!\times\PP(L) \ar[d]_p\INTO^-j& \cH_L\! \ar[d]_\pi\INTO^-\iota& X\times\PP(L) \ar[dl]^\rho \\
X_{L^\perp}\! \INTO^i& X.\!}
\eeq
Notice that $\pi$ has general fibre $\P{\ell-2}$. Over $X_{L^\perp}$ its fibre is $\PP(L)=\P{\ell-1}$.

\subsection{Orlov-type result}
When the baselocus $X_{L^\perp}$ has the expected dimension, the diagram \eqref{peche} induces an inclusion of the derived category of $X_{L^\perp\!}$ into that of the universal hypersurface $\cH_L$ over the linear system.

\begin{prop}[HPD I] \label{hypthm}
Suppose $X_{L^\perp}$ has dimension $\dim X-\ell$. Then
$$
j_*p^*\colon D\big(X_{L^\perp\!}\big)\To D(\cH_L)
$$
is a full and faithful embedding. So is $\pi^*$, and together these give a semi-orthogonal decomposition
\beq{sodd}
D(\cH_L)=\Big\langle D\big(X_{L^\perp\!}\big),\ \pi^*D(X)(0,1),\ \ldots\ ,\ \pi^* D(X)(0,\ell-1)\Big\rangle.
\eeq
\end{prop}

\noindent Here by $(i,j)$ we mean the twist by $\cO(i,j)$, the restriction of $\cO_X(i)\boxtimes\cO_{\PP(V)}(j)$ to $\cH_L\subset X\times\PP(L)$.

\begin{proof}
The baselocus $X_{L^\perp}\!\subset X$ is cut out by sections of $\cO_X(1)$, one for each element of a basis of $L$. Invariantly, it is cut out by the section
$$
\sigma\,\in\,H^0\big(\cO_X(1)\otimes L^*\big)\ \cong\ 
\Hom\big(L,H^0(\cO_X(1))\big)
$$
corresponding to the inclusion $L\subseteq V\subseteq H^0(\cO_X(1))$.
By the assumption on expected dimensions, $\sigma$ is a regular section of $\cO_X(1)\otimes L^*$.

Pulling $\cO_X(1)\otimes L^*$ back to $X\times\PP(L)$, it sits inside an exact sequence
\beq{offi}
0\To\cO_X(1)\boxtimes\Omega_{\PP(L)}(1)\To
\cO_X(1)\boxtimes L^*\Rt{\mathrm{ev}}\cO_X(1)\boxtimes\cO_{\PP(L)}(1)\To0
\eeq
given by the (dual) Euler sequence on $\PP(L)$ tensored by $\cO_X(1)$.  In \eqref{offi} the section $\sigma$ projects to the section
$$
\mathrm{ev}\circ\sigma\,\in\,H^0\big(\cO_{X\times\PP(L)}(1,1)\big)\ \cong\ \Hom\big(L,H^0(\cO_X(1))\big)
$$
which also corresponds to the inclusion $L\subseteq H^0(\cO_X(1))$. Its zero locus is therefore $\cH_L\subset X\times\PP(L)$, on restriction to which $\sigma$ lifts canonically to a section $\widetilde\sigma$ the kernel of \eqref{offi},
$$
\widetilde\sigma\,\in\,H^0\big(\Omega_{\PP(L)}(1,1)|_{\cH_L}\big).
$$
Again by the assumption on dimensions this is a regular section cutting out $X_{L^\perp\!}\times\PP(L)\subset\cH_L$, so the normal bundle of $j$ is
\beq{NNN}
N_j\,\cong\,\cO_{X_{L^\perp}}\!(1)\boxtimes\Omega_{\PP(L)}(1).
\eeq
\smallskip

Next we show that the counit
\beq{coun}
j^*j_*\To\id
\eeq
induces an isomorphism
\beq{isom}
R\Hom(j^*j_*p^*E,p^*E)\oT R\Hom(p^*E,p^*E)
\eeq
for any $E\in D(X_{L^\perp\!})$. Since the first term is $R\Hom(j_*p^*E,j_*p^*E)$ and the second is $R\Hom(E,E)$, this will prove that $j_*p^*$ is fully faithful.

By a standard Koszul computation of (the cohomology sheaves of) the Fourier-Mukai kernel of $j^*j_*$, we see the cone on \eqref{coun} is an iterated extension of the functors
$$
\Lambda^rN_j^*[r]\otimes(\ \cdot\ ), \quad 1\le r\le\ell-1.
$$
Therefore the cone on \eqref{isom} is an extension of the groups
$$
R\Hom\big(\Lambda^rN_j^*\otimes p^*E,p^*E\big)[-r]\ \cong\ 
R\Hom\big(E,E\otimes p_*(\Lambda^rN_j)\big)[-r].
$$
But by \eqref{NNN},
$$
p_*(\Lambda^rN_j)\ =\ p_*\Big(\cO_{X_{L^\perp}}\!(r)\boxtimes\Omega^r_{\PP(L)}(r)\Big)\,=\,0,
$$
since $\Omega^r_{\P{l-1}}(r)$ is acyclic for $1\le r\le\ell-1$. In particular \eqref{isom} is an isomorphism and $j_*p^*$ is fully faithful as claimed. \medskip

Next we check that $\pi^*D(X)(0,k)$ is in the left orthogonal to $j_*p^*D\big(X_{L^\perp}\big)$ for $1\le k\le\ell-1$. Picking $E\in D\big(X_{L^\perp\!}\big)$ and $F\in D(X)$,
\begin{align}\nonumber
R\Hom\big(\pi^*F(0,k),j_*p^*E\big)=R\Hom&\big(j^*\pi^*F(0,k),p^*E\big) \\ =R\Hom\big(p^*i^*F(0,k),\,&p^*E\big)=R\Hom\big(i^*F,E\otimes p_*\cO(0,-k)\big).\label{8}
\end{align}
This vanishes because $p_*\cO(0,-k)=0$. \medskip

Now we check that $\pi^*D(X)(0,k)$ is in the left orthogonal to $\pi^*D(X)(0,n)$ for $1\le n<k\le\ell-1$. So we pick $E,F\in D(X)$ and compute
\beqa
R\Hom\big(\pi^*F(0,k),\pi^*E(0,n)\big) &=&
R\Hom\big(\iota^*\rho^*F(0,k),\iota^*\rho^*E(0,n)\big) \\
&=& R\Hom\big(\rho^*F,\iota_*\iota^*\rho^*E(0,n-k)\big).
\eeqa
By the exact triangle $G(-1,-1)\to G\to \iota_*\iota^*G$ with $G=\rho^*E(0,n-k)$
it is sufficient to show
$$
R\Hom\big(\rho^*F,\rho^*E(0,n-k)\big)\,=\,0\,=\,
R\Hom\big(\rho^*F,\rho^*E(-1,n-k-1)\big),
$$
which follows from $\rho_*\cO(0,n-k)=0=\rho_*\cO(0,n-k-1)$.
\medskip

The same argument with $n=k$ shows that each $\pi^*D(X)(0,k)$ is fully faithfully embedded in $D(\cH_L)$. \medskip

\noindent\textbf{Generation.} Finally we show our semi-orthogonal collection spans $D(\cH_L)$ by showing its left orthogonal is empty. So take
\beq{ta}
E\in{}^\perp\Big\langle D\big(X_{L^\perp\!}\big),\ \pi^*D(X)(0,1),\ \ldots\ ,\ \pi^* D(X)(0,\ell-1)\Big\rangle.
\eeq
By the same Nakayama lemma argument as at the end of Section \ref{bile}, to show $E=0$ it is sufficient to show that its derived restriction $E_x:=E|_{\pi^{-1}(x)}$ to any fibre of $\pi\colon\cH_L\to X$ is zero.

First suppose that $x\not\in X_{L^\perp}$, so that $\pi^{-1}(x)\cong\P{\ell-2}$ and $\pi$ is flat nearby. Thus the pushforward to $\cH_L$ of $\cO_{\pi^{-1}(x)}$ is $\pi^*$ applied to $\cO_x\in D(X)$, and
$$
R\Hom(E_x,\cO(k))=R\Hom\!\big(E,(\pi^*\cO_x)(k)\big)=0
$$
by \eqref{ta}.
Since the sheaves $\cO(k),\ 1\le k\le\ell-1,$ span $D(\P{\ell-2})$ by \eqref{Beil} it follows that $E_x=0$, as required.

\medskip Now take $x\in X_{L^\perp\!}$, so that $\pi^{-1}(x)\cong\P{\ell-1}$. Since $\pi$ is no longer flat near $\pi^{-1}(x)$ we instead used diagram \eqref{peche} to compute $\pi^*\cO_x$:
\beq{croi}
\pi^*\cO_x=\iota^*\rho^*\cO_x=\iota^*\cO_{\{x\}\times\PP(L)}=\iota^*\iota_*j_*\cO_{\{x\}\times\PP(L)}.
\eeq
(Here we are suppressing two different pushforward maps from $\{x\}\times\PP(L)$ into $X_{L^\perp\!}\times\PP(L)$ and $X\times\PP(L)$, which differ by $\iota_*j_*$.)
Since $\iota\colon\cH_L\into X\times\PP(L)$ is a $(1,1)$ divisor, we have an exact triangle $\id(-1,-1)[1]\to\iota^*\iota_*\to\id$. Applied to \eqref{croi} this gives
$$
j_*\cO_{\{x\}\times\PP(L)}(-1,-1)[1]\To\pi^*\cO_x\To j_*\cO_{\{x\}\times\PP(L)}.
$$
Tensoring with $\cO(0,k)$ and applying $R\Hom(E,\ \cdot\ )$ gives the exact triangle
\beq{zoro}
R\Hom\!\big(E_x,\cO_{\PP(L)}(k-1)\big)[1]\To0\To R\Hom\!\big(E_x,\cO_{\PP(L)}(k)\big)
\eeq
for $1\le k\le\ell-1$, because $R\Hom(E,(\pi^*\cO_x)(k))=0$ by assumption \eqref{ta}. But \eqref{ta} also gives
$$
0=R\Hom(E,j_*p^*\cO_x)=R\Hom\!\big(E,\cO_{\{x\}\times\PP(L)}\big)=R\Hom
\!\big(E_x,\cO_{\PP(L)}\big).
$$
Combined with \eqref{zoro} this gives the vanishing
$$
R\Hom\!\big(E_x,\cO_{\PP(L)}(k)\big)=0,\ \ 0\le k\le\ell-1,
$$
which by Beilinson \eqref{Beil} implies $E_x=0$.
\end{proof}

\noindent\textbf{Digression: the Cayley trick.}
Above we picked $\ell$ sections of the same line bundle $\cO_X(1)$ and related their complete intersection to their universal hyperplane in $\P{\ell-1}\times X$.

For more general local complete intersections we might consider $\ell$ sections of $\ell$ different line bundles $L_1,\ldots, L_\ell$, or more generally a section $s$ of a rank $\ell$ bundle $F$. The natural analogue of the universal hyperplane is then given by the zero locus $H_s$ of the section that $s$ defines of the $\cO(1)$ line bundle on the projective bundle
\beq{ppf}
\PP(F)\To X.
\eeq
Though \eqref{ppf} is no longer usually a product\footnote{When $F=\bigoplus_{i=1}^{\ell}\cO_X(d_i)$ it would be nice to work on the product of $X$ with the \emph{weighted} projective space $\PP(d_1,\ldots,d_\ell)$ to get a weighted version of HPD, but I have not been able to make this work.} $\P{\ell-1}\times X$, there is still an HPD story. The hyperplane $H_s\subset\PP(F)$ is generically a $\P{\ell-2}$-bundle, but a $\P{\ell-1}$-bundle over the zeros $Z(s)\subset X$ of $s\in H^0(F)$. This fact is called the ``Cayley trick" in \cite{IM}, where it is used to relate the cohomology of $H_s$ and the local complete intersection $Z(s)$, and categorified in \cite{KKLL}, showing it gives an embedding of derived categories. These follow directly from HPD I applied to the variety $\PP(F)$ and its $\cO(1)$ linear system, i.e. applied to
$$
\PP(F)\To\PP\big(H^0(\cO_\PP(1))^*\big)=\PP\big(H^0(X,F)^*\big).
$$

\section{Homological projective duality II}
\subsection{Lefschetz collections} \label{sods}

HPD I \eqref{hypthm} is really the whole of homological projective duality for the ``stupid Lefschetz decomposition" \cite[Proposition 9.1]{HPD}. However, it is not yet much of a duality between $X_{L^\perp}$ and $\cH_L$ because the latter's derived category is always bigger, containing that of the former. Kuznetsov gives a way to produce more interesting, smaller, examples inside $D(\cH)$.  When $D(X)$ admits certain semi-orthogonal decompositions he uses them to remove some of the copies of $D(X)$ from \eqref{sodd}, refining Proposition \ref{hypthm} by replacing $D(\cH_L)$ by its ``\emph{interesting part}".\medskip

\noindent\textbf{Example.}
The prototypical example is $X=\P{n-1}$. On passing to a degree-$d$ hypersurface $H\subset X$, an easy calculation shows that the right hand end $\cO(d),\ldots,\cO(n-1)$ of the Beilinson semi-orthogonal decomposition \eqref{Beil} remains an exceptional semi-orthogonal collection on restriction to $H$, but one cannot go any further. So Kuznetsov defines the ``interesting part" $\cC_H$ of $D(H)$ to be its right orthogonal:
\begin{eqnarray*}
\cC_H \!&:=&\! \Langle\cO_H(d),\cO_H(d+1),\ldots,\cO_H(n-1)\Rangle^\perp \\
\!&=&\! \big\{E\in D(H)\colon R\Hom(\cO(i),E)=0\ \,\mathrm{for}\ i=d,\ldots,n-1\big\}.
\end{eqnarray*}
This is a $\C$-linear triangulated category which is amazingly, \emph{fractional Calabi-Yau} \cite{Kufrac} --- some power of its Serre functor is a shift $[k]$ (cf. Remark \ref{intCY}). Interesting examples include
\begin{itemize}
\item Smooth even dimensional quadrics \cite{Ka, Kuq}, where $\cC_H$ is generated by two exceptional bundles (the ``spinor bundles") which are orthogonal to each other. Thus $\cC_H\cong D(\mathrm{pt}\sqcup\mathrm{pt})$. In families this leads to double covers of linear systems of quadrics --- see Section \ref{d=2}.
\item Smooth cubic fourfolds \cite{KuC}, where $\cC_H$ is the derived category of a K3 surface, noncommutative in general --- see Section \ref{d=3}.
\end{itemize}
One can put these categories $\cC_H$ together over the $\PP\big(H^0(\cO_{\P{n-1}}(d))\big)$ family of all $H$s. So we set $V=H^0(\cO_{\P{n-1}}(d))$ and let $\cH_L\subset\P{n-1}\times\PP(L)$ be the universal hypersurface as before, for any $L\subseteq V$. Then define
$$
\cC_{\cH_L}\,:=\,\Langle D(\PP(L))(d,0),\,D(\PP(L))(d+1,0),\,\ldots,\,D(\PP(L))(n-1,0)\Rangle^\perp.
$$
Putting $L=V$ gives $\cC_{\cH}$, which is what we will define to be the HP dual of $(\P{n-1},\cO(d))$ with its exceptional collection \eqref{Beil}. Below we will be able to refine HPD I \eqref{hypthm} by replacing $D(\cH_L)$ with the smaller subcategory $\cC_{\cH_L}$.

(Kuznetsov \cite{HPD} would ask further that the HP dual be \emph{geometric}: roughly that we have a $D(\PP(V))$-linear equivalence $D(Y)\cong\cC_{\cH}$ for some variety $Y\to\PP(V)$. In these notes, however, we consider this as a separate --- very important --- step, which we address in the examples of Section \ref{egs}.)\medskip

\noindent\textbf{General case.} For a more general smooth polarised variety $(X,\cO_X(1))$ Kuznetsov replaces \eqref{Beil} with what he calls a \emph{Lefschetz decomposition} of $D(X)$. For simplicity we restrict attention to a \emph{rectangular} Lefschetz decomposition of $D(X)$, which means an \emph{admissible}\footnote{Admissible means there exists a left adjoint to the inclusion functor. This is equivalent to the existence of a right adjoint, by Serre duality.} subcategory $\cA\subseteq D(X)$ generating a semi-orthogonal decomposition
\beq{rect}
D(X)\,=\,\Langle\cA,\cA(1),\ldots,\cA(i-1)\Rangle.
\eeq
For instance the example above corresponds to taking $(X,\cO_X(1))$ to be $(\P{n-1},\cO_{\P{n-1}}(d))$ with $i=n/d$ and
$$
\cA=\Langle\cO_{\P{n-1}},\cO_{\P{n-1}}(1),\ldots,\cO_{\P{n-1}}(d-1)\Rangle.
$$
(This can be generalised to $d\!\nmid\!n$ by using the non-rectangular Lefschetz decomposition $\Langle\cA,\cA(1),\ldots,\cA(i-2),\cA'(i-1)\Rangle$, where $\cA'\subset\cA$ is the span $\langle\cO_{\P{n-1}},\cO_{\P{n-1}}(1),\ldots,\cO_{\P{n-1}}(r-1)\rangle$ for $i=\lceil\frac nd\rceil$ and $n=(i-1)d+r$.)
Another example is $\Gr(2,2n+1)$ with $i=2n+1$ and
$$
\cA=\Langle S^{n-1}A,\ldots,A,\cO\Rangle,
$$ where $A\to\Gr(2,2n+1)$ is the universal rank 2 subbundle \cite{KuGr}. For many other examples see \cite{ICM}.

For a given hyperplane $H\subset X$ we get a semi-orthogonal decomposition
\beq{DH}
D(H)\,=\,\Langle\cC_H,\cA(1),\cA(2),\ldots,\cA(i-1)\Rangle
\eeq
by setting $\cC_H$ to be the right orthogonal of $\cA(1),\ldots,\cA(i-1)$ in $D(H)$. (These form a semi-orthogonal exceptional collection by an application the following Lemma to $\PP(L)\subset \PP(V)$, where $L=\langle s_H\rangle$ is the one dimensional subspace spanned by the section $s_H\in V$ cutting out $H$.)
In Remark \ref{intCY} we will see an easy proof that $\cC_H$ is Calabi-Yau under certain conditions.

\begin{lem} \label{vanish}
The functor $\pi^*\colon D(X)\to D(\cH_L)$ restricted to $\cA\subset D(X)$ defines an embedding $\cA\subset D(\cH_L)$. Then $R\Hom_{\cH_L}\!\big(\cA(\alpha,\beta),\cA\big)$ vanishes in each of the following cases:
\begin{itemize}
\item $0<\alpha<i-1$,
\item $0<\beta<\ell-1$,
\item $\alpha=0,\ \beta=\ell-1$,
\item $\alpha=i-1,\ \beta=0$.
\end{itemize}
In particular, $\cA(1)\boxtimes D(\PP(L)),\ldots,\cA(i-1)\boxtimes D(\PP(L))$ form a semi-orthogonal collection in $D(\cH_L)$.
\end{lem}

\begin{proof}
Since $\cH_L\subset X\times\PP(L)$ is a (1,1) divisor, we have an exact triangle of functors $\id(-1,-1)\to\id\to\iota_*\iota^*$ on $D(X\times\PP(L))$. It follows that
\beqa
R\Hom_{\cH_L}\!\big(\pi^*\cA(\alpha,\beta)
,\pi^*\cA\big) &=&
R\Hom_{\cH_L}\!\big(\iota^*\rho^*\cA(\alpha,\beta)
,\iota^*\rho^*\cA\big) \\ 
&=& R\Hom_{X\times\PP(L)}\!\big(\rho^*\cA(\alpha,\beta)
,\iota_*\iota^*\rho^*\cA\big)
\eeqa
is the cone on
$$
R\Hom_{X\times\PP(L)}\big(\rho^*\cA(\alpha+1,\beta+1),\rho^*\cA\big)\To
R\Hom_{X\times\PP(L)}\big(\rho^*\cA(\alpha,\beta),\rho^*\cA\big).
$$
By the K\"unneth formula this equals
\begin{align}\nonumber
R\Hom\_X\!\big(\cA(\alpha+1),\cA\big)\otimes&R\Gamma\big(\cO_{\PP(L)}(-\beta-1)\big)\\ & \To R\Hom\_X\!\big(\cA(\alpha),\cA\big)\otimes R\Gamma\big(\cO_{\PP(L)}(-\beta)\big). \label{number}
\end{align}
The first term vanishes for $0\le\alpha<i-1$, the second for $0\le\beta<\ell-1$, the third for $0<\alpha\le i-1$ and the fourth for $0<\beta\le\ell-1$. This proves full and faithfullness for $\alpha=0=\beta$ (we already knew it for $\ell>1$ by Proposition \ref{hypthm}, but not for $\ell=1$) and gives vanishing for $(\alpha,\beta)$ on the list. Recalling the Beilinson semi-orthogonal decomposition \eqref{Beil} gives the claimed semi-orthogonal collection in $D(\cH_L)$.
\end{proof}

Therefore for any linear subspace $L\subseteq V$ we can define
\beq{bath}
\cC_{\cH_L}\,:=\,\Langle\cA(1)\boxtimes D(\PP(L)),\,\ldots\,,\,\cA(i-1)\boxtimes D(\PP(L))\Rangle^\perp
\eeq
in $D(\cH_L)$. This defines $\cC_{\cH}$ by taking $L=V$ in \eqref{bath}. It is a $D(\PP(V))$-linear category whose
basechange\footnote{For $\cC_{\cH}$ geometric this is obvious; more generally it is \cite[Proposition 5.1]{KuBC}. As in \cite{KuBC}, basechange means the smallest category containing the derived restrictions of objects in $\cC_{\cH}$ and closed under taking direct summands. (When $\cH_L$ is singular one also has to take a completion.)} to $\PP(L)\subseteq\PP(V)$ is $\cC_{\cH_L}$.

\begin{defn} \label{defff} Fix $X\to\PP(V)$ and a rectangular Lefschetz collection \eqref{rect}. We call the resulting category $\cC_{\cH}$ the \emph{homological projective dual} of $D(X)$.
\end{defn}

Using the Beilinson semi-orthogonal decomposition \eqref{Beil} to further split the subcategories $\cA(k)\boxtimes D(\PP(L))$ we illustrate HPD I \eqref{sodd} via the white boxes in Figure \ref{fig7} below. 

\begin{figure}[htbp]
\begin{center}
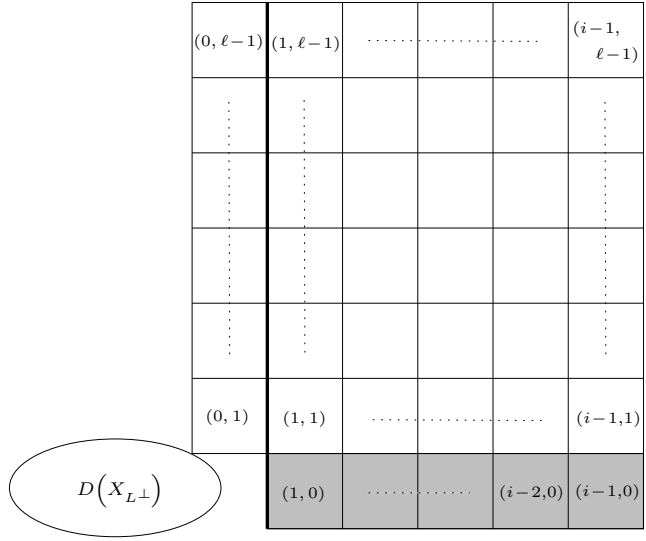
\caption{\footnotesize The \emph{white} boxes and ellipse illustrate the semi-orthogonal decomposition  $D(\cH_L)=$ 
$\Langle D\big(X_{L^\perp\!}\big),\ \pi^*D(X)(0,1),\ \ldots\ ,\ \pi^* D(X)(0,\ell-1)\Rangle$
of \eqref{sodd}, with $\cA(\alpha,\beta)$ in the $(\alpha,\beta)\;$th box. \emph{All} boxes to the right of the thick line represent ${}^\perp\cC_{\cH_L}$ \eqref{bath}. } \label{fig7}
\end{center}
\end{figure}
To the right of the thick line the white and grey boxes represent the subcategory $\Langle\cA(1)\boxtimes D(\PP(L)),\,\ldots\,,\,\cA(i-1)\boxtimes D(\PP(L))\Rangle$ whose right orthogonal defines $\cC_{\cH_L}$ \eqref{defff}.  Notice its semi-orthogonal decomposition is not yet compatible with the decomposition of HPD I; in particular each grey box does not fit inside any one unique white box of $D(\cH_L)$.
 
The passage from HPD I to HPD II can be paraphrased loosely (up to mutations specified in the proof) in terms of this diagram as follows. For $i\le\ell$, the $i-1$ grey boxes fit into the top $i-1$ white boxes in the left hand column, and the remaining white boxes below it are then $\cC_{\cH_L}$, which we see contains $D\big(X_{L^\perp\!}\big)$ (or equals it for $i=\ell$).

When $i\ge\ell$ the inclusion goes in the opposite direction. The first $\ell-1$ grey boxes fit into the white boxes of the left hand column. The remaining $i-l$ grey boxes then mutate into $D\big(X_{L^\perp\!}\big)$, which therefore contains $\cC_{\cH_L}$ as their right orthogonal.

\subsection{HPD II}
Pick a linear subspace $L\subseteq V$. Equivalently we fix a linear subspace $L^\perp\!\subset V^*$. Projecting Proposition \ref{hypthm} into $\cC_{\cH_L}$ (for $\ell\ge i$) or into $D\big(X_{L^\perp\!}\big)$ (for $\ell\le i$) gives a close relationship between the derived category of the resulting linear section $X_{L^\perp}$ of $X$ (the baselocus of the linear system $\PP(L)$) and the orthogonal linear section $\cC_{\cH_L}$ of the HP dual $\cC_{\cH}$.

\begin{thm} \label{hpdthm} Projecting the subcategories of $D(\cH_L)$ of Proposition \ref{hypthm} into $\cC_{\cH_L}$ or $D\big(X_{L^\perp\!}\big)$ gives the following semi-orthogonal decompositions.
\begin{itemize}
\item If $\ell>i$ then $\cC_{\cH_L}=\Langle D\big(X_{L^\perp\!}\big),\ \cA(0,1),\ \cA(0,2),\,\ldots\,,\,\cA(0,\ell-i)\Rangle.$
\item If $\ell=i$ then $\cC_{\cH_L}\cong D\big(X_{L^\perp\!}\big).$
\item If $\ell<i$ then $D\big(X_{L^\perp\!}\big)=\Langle\cC_{\cH_L},\ \cA(1,0),\ \cA(2,0),\,\ldots\,,\,\cA(i-\ell,0)\Rangle.$
\end{itemize}
\end{thm}

\begin{rmk}
In particular the first result with $L=V$ shows that $\cC_{\cH}$ has a rectangular Lefschetz collection
$$
\cC_{\cH}\,=\,\Langle\cA(0,0),\ \cA(0,1),\,\ldots\,,\,\cA(0,j-1)\Rangle, \qquad j:=\dim V-i,
$$
similar to that of $D(X)$ \eqref{rect}. (Here we have used the invariance of $\cC_{\cH}$ under tensoring by $\cO(0,-1)$.) Applying the same construction to it,\footnote{In other words we replace the $D(\PP(V^*))$-linear category $D(X)$ by the $D(\PP(V))$-linear category $\cC_\cH$ throughout Section 3.1. With some work one can make sense of this \cite{Pe}.}
Theorem \ref{hpdthm} shows its HP dual is $D(X)$, so HPD is indeed a duality.
\end{rmk}

\begin{rmk} \label{HPDIII}
Picking our functors slightly differently, the third result can be rewritten
$$
D\big(X_{L^\perp\!}\big)\,=\,\Langle\cC_{\cH_L},\ \cA(\ell),\ \cA(\ell+1),\,\ldots\,,\,\cA(i-1)\Rangle.
$$ 
This should be compared to \eqref{DH}. The moral is that when $\ell\le i$, to pass from $D(X)$ to a codimension-$\ell$ linear section $D\big(X_{L^{\perp\!}}\big)$ we
\begin{itemize}
\item lose the first $\ell$ copies of $\cA$ in the Lefschetz decomposition, and 
\item gain the (restriction to $X_{L^\perp\!}$ of) the $\PP(L)=\P{\ell-1}$ family of categories $\cC_H$, where $H$ runs through the hyperplanes containing $X_{L^\perp}$.
\end{itemize}
\end{rmk}

\begin{rmk} \label{intCY}
Mutating $\cA$ to the right hand end of the Lefschetz collection shows that $\cA(i)=\cA\otimes K_X^{-1}$, so $K_X$ is close to being $\cO_X(-i)$. If in fact $K_X\cong\cO_X(-i)$, the complete intersection $X_{L^\perp}$ of $i$ hyperplanes is Calabi-Yau, so Theorem \ref{hpdthm} shows that the corresponding $\cC_{\cH_L}$ is also Calabi-Yau.  

In the geometric case this gives a cheap proof that its fibre $\cC_H$ over any hyperplane $H\in\PP(L)=\PP^{i-1}$ is also Calabi-Yau.

To deduce this last step more generally note the inclusion $I\colon H\into\cH_L$ has trivial normal bundle, so $I_!=I_*[-(\ell-1)]$. Since the Serre functor of $\cC_{\cH_L}$ is a shift $[N]$ and $I_*,\,I^*$ preserve the subcategories $\cC$, we have
\begin{multline*}
R\Hom(I^*F,I^*G)\cong R\Hom(F,I_*I^*G)\cong R\Hom(I_*I^*G,F)^\vee[-N] \\ \cong R\Hom(I_!I^*G,F)^\vee[-N+\ell-1]\cong R\Hom(I^*G,I^*F)^\vee[-N+\ell-1]
\end{multline*}
for $F,G\in\cC_{\cH_L}\subset D(\cH_L)$.
But objects of the form $I^*F$ split generate $\cC_H$ 
\cite[Proposition 5.1]{KuBC} so this gives Serre duality with Serre functor a shift for the whole category. That is, $\cC_H$ is Calabi-Yau.
\end{rmk}

We split the proof of Theorem \ref{hpdthm} into two cases.

\begin{proof}[Proof of $\ell\ge i$ case.]
The projection $\pi_L\colon D(\cH_L)\to\cC_{\cH_L}$ is the left adjoint to the inclusion $\cC_{\cH_L}\into D(\cH_L)$. It is given by left mutation past the semi-orthogonal collection of Lemma \ref{vanish} (in reverse order):
$$
\pi\_L\ =\ \LL_{\cA(1)\boxtimes D(\PP(L))}\circ\ldots\circ\LL_{\cA(i-1)\boxtimes D(\PP(L))}\,.
$$
Restricted to any of the subcategories
$$
D\big(X_{L^\perp\!}\big),\ \cA(0,1),\ \cA(0,2),\,\ldots\,,\,\cA(0,\ell-i)
$$
given to us by Proposition \ref{hypthm}, we would like to show that $\pi\_L$ is fully faithful, and that it preserves the semi-orthogonal condition between them.

Now $\pi\_L$ commutes with $\otimes\cO(0,1)$ since \eqref{bath} does. Therefore it is sufficient to prove
\beq{keep}
R\Hom(\pi\_La,\pi\_Lb)=R\Hom(a,b)
\eeq
in each of the cases
\begin{enumerate}
\item[(i)] $a\in\cA(0,k),\ 1\le k\le\ell-i,\ \ b\in\cA(0,1)$,
\item[(ii)] $a\in\cA(0,k),\ 1\le k\le\ell-i,\ \ b\in D\big(X_{L^\perp\!}\big)$,
\item[(iii)] $a,b\in D\big(X_{L^\perp\!}\big)$.
\end{enumerate}
Since $\pi_L$ is the left adjoint of the inclusion $\cC_{\cH_L}\into D(\cH_L)$,
the left hand side of \eqref{keep} equals $R\Hom(a,\pi\_Lb)$.
To show this also equals the right hand side it is sufficient
to show that
\beq{ask}
R\Hom\!\big(a,\mathrm{Cone}\;(b\to\pi\_Lb)\big)=0
\eeq
in each of the cases (i),\,(ii) and (iii).

\medskip \noindent\textbf{Case (i).} To begin with we take $b\in\cA(0,1)$ and analyse $\pi\_Lb$, inspired by \cite[Lemma 5.6]{HPD}. We first left mutate past the category $\cA(i-1)\boxtimes D(\PP(L))$, which has a Beilinson semi-orthogonal decomposition into subcategories
$$
\cA(i-1)\boxtimes D(\PP(L))\ =\ \Langle\cA(i-1,0),\,\ldots\,,\,\cA(i-1,\ell-1)\Rangle
$$
of which only $\cA(i-1,0)$ has Homs to $b$ by Lemma \ref{vanish}. Thus
$$
b^{(1)}:=\LL_{\cA(i-1)\boxtimes D(\PP(L))}\,b\,=\,\LL_{\cA(i-1,0)}\,b.
$$

For $\cA(i-2)\boxtimes D(\PP(L))$ we use the \emph{shifted} Beilinson semi-orthogonal decomposition
$$
\Langle\cA(i-2,-1),\,\ldots\,,\,\cA(i-2,\ell-2)\Rangle.
$$
By Lemma \ref{vanish}, of these subcategories only $\cA(i-2,-1),\,\cA(i-2,0)$ have Homs to $b^{(1)}\in\Langle\cA(i-1,0),\cA(0,1)\Rangle$, so
$$
b^{(2)}:=\LL_{\cA(i-2)\boxtimes D(\PP(L))}\,b^{(1)}\,=\,\LL_{\cA(i-2,-1)}\LL_{\cA(i-2,0)}\,b^{(1)}.
$$

We continue inductively. After the $(k-1)\;$th stage we have mutated $b$ into $b^{(k-1)}$, lying in the category generated by $\cA(0,1)$ and
\beq{list1}
\cA(\alpha,\beta),\ \ \alpha\in[i-k+1,i+\beta-1],\ \beta\in[2-k,0],
\eeq
shaded grey in Figure \ref{fig1} below.

\begin{figure}[htbp]
\begin{center}
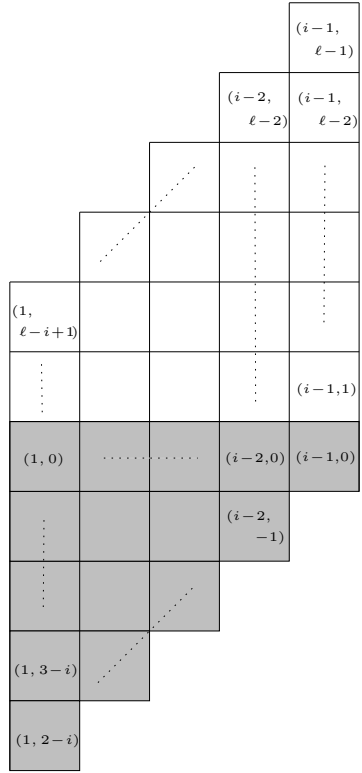
\caption{\small The semi-orthogonal collection $\cA(1)\boxtimes D(\PP(L)),\,\ldots,\,\cA(i-1)\boxtimes D(\PP(L))$, with $\cA(\alpha,\beta)$ in the $(\alpha,\beta)\;$th box. We left mutate $b$ past only the shaded subcategories.}
\label{fig1}
\end{center}
\end{figure}

To mutate past $\cA(i-k)\boxtimes D(\PP(L))$, we use the following shifted Beilinson semi-orthogonal decomposition,
$$
\Langle\cA(i-k,-k+1),\,\ldots\,,\,\cA(i-k,\ell-k)\Rangle.
$$
By Lemma \ref{vanish} only the first $k$ subcategories
have Homs to $b^{(k-1)}$, so
$$
b^{(k)}:=\LL_{\cA(i-k)\boxtimes D(\PP(L))}\,b^{(k-1)}\,=\,\LL_{\cA(i-k,-k+1)}\circ\ldots\circ\LL_{\cA(i-k,0)}\,b^{(k-1)}.
$$

Therefore finally we see that $b^{(i-1)}=\pi\_Lb$ is an iterated extension of $b$ and objects in the categories
\beq{guff}
\cA(\alpha,\beta),\ \ \alpha\in[1,i+\beta-1],\ \beta\in[2-i,0].
\eeq
Thus Cone$(b\to\pi\_Lb)$ lies in the span of \eqref{guff}. By Lemma \ref{vanish} this has no Homs from $a\in\cA(0,k),\ 1\le k\le\ell-i$. This proves \eqref{ask} in case (i).

\medskip \noindent\textbf{Case (ii).} Now take $b$ in the subcategory $j_*p^*D(X_{L^\perp\!})\subset D(\cH_L)$ of Proposition \ref{hypthm}. We saw there that $b$ has no Homs from $\pi^*D(X)(0,k)$ for $1\le k\le\ell-1$. In particular, it has no Homs from $\cA(\alpha,\beta)$ for $1\le\alpha\le i-1,\ 1\le\beta\le\ell-1$. But this is the \emph{only} property of $b$ that we used in the $b\in\cA(0,1)$ case above. So the same working shows again that Cone$(b\to\pi\_Lb)$ lies in the span of \eqref{guff}. By Lemma \ref{vanish} again, this proves \eqref{ask} in case (ii).

\medskip \noindent\textbf{Case (iii).} It also means that to prove \eqref{ask} in case (iii) it is sufficient to show that there are no Homs from $j_*p^*D(X_{L^\perp\!})$ to \eqref{guff}.
But for $E\in D(X_{L^\perp\!})$ and $F\in\cA\subset D(X)$, we have
\begin{eqnarray}
R\Hom(j_*p^*E,\pi^*F &&\hspace{-1cm} (\alpha,\beta))\ =\ R\Hom(p^*E,j^!\pi^*F(\alpha,\beta))
\nonumber \\ &=& R\Hom(p^*E,j^*\pi^*F(\alpha,\beta)\otimes\Lambda^{\ell-1}N_j)[1-\ell]
\nonumber \\ &=& R\Hom\!\big(p^*E,p^*i^*F(\alpha,\beta)\otimes\cO(\ell-1,-1)\big)[1-\ell]
\label{vanny} \end{eqnarray}
using the notation of diagram \eqref{peche} and
$j^!=j^*\otimes\omega_j[1-\ell]=j^*\otimes\bigwedge^{\ell-1}N_j[1-\ell]$, where $N_j\cong\cO_{X_{L^\perp}}\!(1)\boxtimes\Omega_{\PP(L)}(1)$ is the normal bundle to $j$ \eqref{NNN}. Since $p_*\cO(0,\beta-1)=0$ for $-\ell+2\le\beta\le0$ we get vanishing for $\beta\in[-\ell+2,0]\supseteq[-i+2,0]$; in particular for all $(\alpha,\beta)$ on the list \eqref{guff}.

\medskip \noindent\textbf{Generation.} 
Finally we must prove that these subcategories $\pi_LD\big(X_{L^\perp\!}\big)$ and $\pi_L\;\cA(0,1),$ $\ldots\,,\,\pi_L\;\cA(0,\ell-i)$  generate $\cC_{\cH_L}$. It is equivalent to show that
\beq{list}
D\big(X_{L^\perp\!}\big),\ \cA(0,1),\,\ldots\,,\,\cA(0,\ell-i),\ \cA(1)\boxtimes D(\PP(L)),\,\ldots\,,\,\cA(i-1)\boxtimes D(\PP(L))
\eeq
generate $D(\cH_L)$, since $\pi_L$ only altered subcategories on the left of \eqref{list} by terms on the right, leaving the total span unaffected. But the categories 
$$
D\big(X_{L^\perp\!}\big)\ \ \mathrm{and}\ \ \cA(\alpha,\beta),\ \ \ 0\le\alpha\le i-1,\ 1\le\beta\le\ell-1,
$$
generate $D(\cH_L)$ by HPD I (Proposition \ref{hypthm}), so it is enough to show they all lie in the span of \eqref{list}. This is immediate for all but the subcategories
\beq{2ndlist}
\cA(0,\ell-i+1),\,\ldots\,,\,\cA(0,\ell-1).
\eeq
Therefore we will show that the categories \eqref{2ndlist} lie in the span of \eqref{list}. \medskip

We work inductively on the $\cA(0,k)$, starting with $k=\ell-i+1$ and working up to $k=\ell-1$. Since $\cA(0,k)$ lies in the $(k+1)$st term $\pi^*D(X)(0,k)$ of the semi-orthogonal decomposition \eqref{sodd} of HPD I, we may left mutate it past the $k$ terms of \eqref{sodd} preceeding it. We already know all of these terms lie in the span of \eqref{list}; this is immediate for the base case $k=\ell-i+1$ and is part of the induction assumption for larger $k$. After the mutation, $\cA(0,k)$ lands up in the last $\ell-k$ terms of \eqref{sodd} tensored with $K_{\cH_L}\cong K_X(1,1-\ell)$. This is
$$
\Langle D(X)(0,k+1-\ell),\,D(X)(0,k+2-\ell),\,\ldots,\,D(X)\Rangle\ =
\!\!\!\!\!\mathop{\Big\langle\cA(\alpha,\beta)
\Big\rangle_{\!\alpha\in[1,i]}}_{\hspace{2cm}\beta\in[k+1-\ell,0]}\vspace{-7mm}
$$
\beq{label}\eeq
since tensoring with $K_X(1,0)$ does not change $\pi^*D(X)$. So it is sufficient to show that the mutation of $\cA(0,k)$, which we call $M_k$, has $\cA(i,\beta)$-component zero for each $\beta\in[k+1-\ell,0]$; this will imply it lies in 
$$
\cA(1)\boxtimes D(\PP(L)),\,\ldots\,,\,\cA(i-1)\boxtimes D(\PP(L)),
$$
and so in \eqref{list}, as required. In fact we will show more; that the component of $M_k$ in 
\beq{komp}
\cA(\alpha,\beta), \quad \beta\in[k+1-\ell,0],\ \alpha\in[i+\beta,i],
\eeq
is zero. (Though we only care about $\alpha=i$ we need vanishing for the other $\alpha$ to make the induction below work.) As illustrated in Figure \ref{fig2} below, we do this by an increasing induction in $\beta$, running from $k+1-\ell$ to $0$, and --- for each fixed $\beta$ --- a decreasing induction on $\alpha$, running from $i$ to $i+\beta$.

\begin{figure}[htbp]
\begin{center}
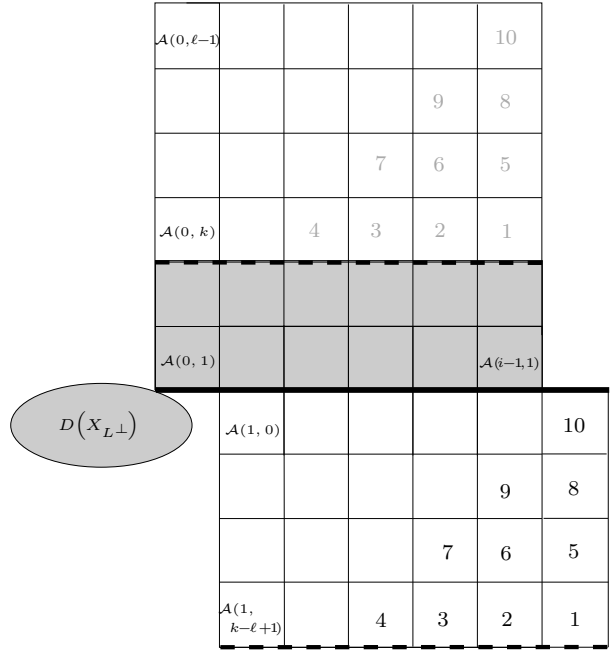
\caption{\small $\cA(0,k)$ is mutated past the shaded boxes into the boxes \eqref{label} below the thick horizontal line. We then show its components in the black numbered boxes vanish in the order indicated, by taking Homs from the corresponding grey numbered boxes.}\label{fig2}
\end{center}
\end{figure}

So we now fix $(\alpha,\beta)$ and prove the vanishing of the $\cA(\alpha,\beta)$-component of $M_k$. We assume by induction the vanishing of $M_k$'s components in \eqref{label} with smaller $\beta$, or the same $\beta$ but bigger $\alpha$. This induction starts with $(\alpha,\beta)=(i,1+k-\ell)$, where the induction assumption is vacuous.

We detect the $\cA(\alpha,\beta)$ component of $M_k$ using Homs from $\cA(\alpha-1,\beta+\ell-1)$. By Lemma \ref{vanish} this has Homs to only
$$
\cA(m,\beta)_{m\in[\alpha,i]}\qquad \mathrm{and}
\mathop{\cA(m,n)_{m\in[\alpha-1,i]}}_{\hspace{22mm}n\in[k-\ell+1,\beta-1]}
$$
of the $\cA(m,n)$s in \eqref{label}. But by the induction step \eqref{komp} the component of $M_k$ is zero in all of these except $\cA(\alpha,\beta)$. So Homs from $\cA(\alpha-1,\beta+\ell-1)$ see only the component in $\cA(\alpha,\beta)$. And they do see this component, by the calculation
$$
R\Hom_{\cH_L}\!\big(\cA(\alpha-1,\beta+\ell-1),\,\cA(\alpha,\beta)\big)\ \cong\  R\Hom_X(\cA,\cA)[2-\ell],
$$
of \eqref{number} (and using $R\Gamma(\cO_{\PP(L)}(-\ell))\cong\C[1-\ell]$). Therefore Homs in $D(\cH_L)$ from $\cA(\alpha-1,\beta+\ell-1)$ detect any $\cA(\alpha,\beta)$ component of $M_k$.

But, by Lemma \ref{vanish} and \eqref{8}, there are no Homs from $\cA(\alpha-1,\beta+\ell-1)$ to either $\cA(0,k)$ or anything it was left mutated past. Thus there are no Homs to $M_k$, which therefore has no component in $\cA(\alpha,\beta)$.
\end{proof}

\begin{proof}[Proof of $\ell<i$ case.] 
This proof is isomorphic to the last one, with $i$ and $\ell$ interchanged and everything transposed. Instead of projecting into
$$\Langle\cA(\alpha,\beta),\ (\alpha,\beta)\in[1,i-1]\times[0,\ell-1]\Rangle^\perp=\,\cC_{\cH_L}$$
we project into 
$$\Langle\cA(\alpha,\beta),\ (\alpha,\beta)\in[0,i-1]\times[1,\ell-1]\Rangle^\perp=\,D\big(X_{L^\perp\!}\big).$$
(Proposition \ref{hypthm} gives the equality.)
We run through the details very briefly.

\medskip
We project $\cC_{\cH_L},\, \cA(1,0),\,\ldots\,,\,\cA(i-\ell,0)$ into $D(X_{L^\perp}\!)$ by the functor
$$
\pi\_{\!L^\perp}=\ \LL_{D(X)(0,1)}\circ\ldots\circ\LL_{D(X)(0,\ell-1)}\colon D(\cH_L)\To D(X_{L^\perp}\!),
$$
cf. Proposition \ref{hypthm}.

Just as in \eqref{ask} this is full and faithful on $\cC_{\cH_L},\, \cA(1,0),\,\ldots\,,\,\cA(i-\ell,0)$, and preserves the semi-orthogonality condition between them, if
\beq{ask2}
R\Hom\!\big(a,\mathrm{Cone}\;(b\to\pi\_{\!L^\perp}b)\big)=0
\eeq
in each of the cases
\begin{enumerate}
\item[(i)] $a\in\cA(k,0),\ 1\le k\le i-\ell,\ \ b\in\cA(1,0)$,
\item[(ii)] $a\in\cA(k,0),\ 1\le k\le i-\ell,\ \ b\in\cC_{\cH_L}$,
\item[(iii)] $a,b\in\cC_{\cH_L}$.
\end{enumerate}

\medskip \noindent\textbf{Case (i).} To begin with we take $b\in\cA(1,0)$ and analyse $\pi\_{\!L^\perp}b\;$. We first left mutate past
$$
D(X)(0,\ell-1)=\Langle\cA(0,\ell-1),\,\ldots\,,\,\cA(i-1,\ell-1)\Rangle
$$
of which only $\cA(0,\ell-1)$ has Homs to $b$ by Lemma \ref{vanish}. Thus
$$
b^{(1)}:=\LL_{D(X)(0,\ell-1)}\,b\,=\,\LL_{\cA(0,\ell-1)}\,b.
$$
For $D(X)(0,\ell-2)$ we use the \emph{shifted} semi-orthogonal decomposition
$$
D(X)(0,\ell-2)=\Langle\cA(-1,\ell-2),\cA(0,\ell-2),\,\ldots\,,\,\cA(i-2,\ell-2)\Rangle.
$$
By Lemma \ref{vanish}, of these subcategories only $\cA(-1,\ell-2),\,\cA(0,\ell-2)$ have Homs to $b^{(1)}\in\Langle\cA(0,\ell-1),\cA(1,0)\Rangle$, so
$$
b^{(2)}:=\LL_{D(X)(0,\ell-2)}\,b^{(1)}\,=\,\LL_{\cA(-1,\ell-2)}\LL_{\cA(0,\ell-2)}\,b^{(1)}.
$$
Proceeding inductively just as before we find that $b^{(\ell-1)}=\pi\_{\!L^\perp}b$ is an iterated extension of $b$ and objects in the categories
\beq{guff2}
\cA(\alpha,\beta),\ \ \alpha\in[2-\ell,0],\ \beta\in[1,\ell+\alpha-1].
\eeq
Thus Cone$(b\to\pi\_{\!L^\perp}b)$ lies in the span of \eqref{guff2}. By Lemma \ref{vanish} this has no Homs from $a\in\cA(k,0),\ 1\le k\le i-\ell$. This proves \eqref{ask2} in case (i).

\medskip \noindent\textbf{Case (ii).} Now take $b$ in the subcategory $\cC_{\cH_L}\subset D(\cH_L)$. By its definition it certainly has no Homs from $\cA(\alpha,\beta)$ for $1\le\alpha\le i-1,\ 1\le\beta\le\ell-1$. But this is the \emph{only} property of $b$ that we used in the $b\in\cA(1,0)$ case above. So the same working shows again that Cone$(b\to\pi\_Lb)$ lies in the span of \eqref{guff2}. By Lemma \ref{vanish} again, this proves \eqref{ask2} in case (ii).

\medskip \noindent\textbf{Case (iii).} It also means that to prove \eqref{ask2} in case (iii) it is sufficient to show that there are no Homs from $\cC_{\cH_L}$ to \eqref{guff2}. Equivalently, by Serre duality we want to show the vanishing of
\beq{trayn}
R\Hom\!\big(\cA(\alpha,\beta)\otimes\omega_{\cH_L}^{-1},\cC_{\cH_L}\big),\quad
\alpha\in[2-\ell,0],\ \beta\in[1,\ell+\alpha-1]
\eeq
Now $\cA\otimes\omega_X^{-1}\cong\cA(i)$ as both are ${}^\perp\Langle\cA(1),\ldots,\cA(i-1)\Rangle$, so
$$
\cA\otimes\omega^{-1}_{\cH_L}=\cA\otimes\big(\omega^{-1}_X\boxtimes\omega^{-1}_{\PP(L)}\big)(-1,-1)=
\cA(i-1,\ell-1).
$$
Therefore \eqref{trayn} is $R\Hom\!\big(\cA(\alpha+i-1,\beta+\ell-1),\cC_{\cH_L}\big)$, which vanishes for $\alpha\in[2-\ell,0]\subset[2-i,0]$ by the definition of $\cC_{\cH_L}$.

\medskip \noindent\textbf{Generation.}
This is very similar to generation in the previous case. Since we have now mutated quite enough and demonstrated to death all of the techniques required, we leave this to the fanatical reader.
\end{proof}

\section{Examples}\label{egs}
To find an HP dual $Y\to\PP(V)$ of some variety $X\to\PP(V^*)$ involves finding a Fourier-Mukai kernel $\mathcal U$ over $Y\times_{\PP(V)}\cH$ whose induced Fourier-Mukai functor $D(Y)\to D(\cH)$  is an equivalence onto $\cC_{\cH}$. The archetypal example is \eqref{twis} below. In general it is very hard to achieve, especially over the discriminant locus of $\cH\to\PP(V)$. By now, however, many examples have been worked out.

\subsection[Example]{Even dimensional quadrics}\label{d=2}
The classical and motivating example of HPD is that of $\P{2n+1}$ with the $\cO(2)$ line bundle (i.e. of the Veronese embedding $\P{2n+1}\to\PP^{(2n+1)(n+2)}$). This has the standard rectangular Lefschetz decomposition $D(\P{2n+1})=\Langle\cA,\cA(2),\ldots,\cA(2n)\Rangle$ with $\cA=\langle\cO,\cO(1)\rangle$.

So the first thing to consider is a smooth even dimensional quadric hypersurface $H\subset\P{2n+1}$. Then the interesting part $\cC_H=\Langle\cO_H,\ldots,\cO_H(2n-1)\Rangle^\perp$  of its derived category is
$$
\cC_H\,=\,\Langle A,B\Rangle,
$$
where $A$ and $B$ are the \emph{spinor bundles} defined using the Clifford algebra of the quadratic form defining $H$. For instance when $n=1$ we have $H\cong\P1\times\P1$ and the spinor bundles are the line bundles $\cO_H(-1,0)$ and $\cO_H(0,-1)$. For $n=2$ the quadric $H\subset\P5$ is Gr$\,(2,4)$ in its Pl\"ucker embedding, and the spinor bundles are the universal subbundle and dual of the universal quotient bundle. In general all we need to know is that they are exceptional and mutually orthogonal, so
$$
\cC_H\,\cong\,D(\mathrm{pt}\sqcup\mathrm{pt})
$$
is geometric --- it is the derived category of 2 disjoint points. Varying $H$ through the linear system of quadrics $(\P{2n+1})^*$, we therefore expect there to be an HP dual given by a double cover
$$
\xymatrix{Y \ar@{->>}[r] & (\P{2n+1})^*.}
$$

In fact, as $H$ becomes singular when the quadratic form defining it drops rank by 1, the spinor bundles become a single spinor sheaf, so this double cover $Y$ branches over the degree $2n+2$ hypersurface of degenerate quadrics in $(\P{2n+1})^*$. For one paragraph, let us work away from the loci where the rank drops further. Then we have a universal spinor sheaf on the universal quadric hypersurface locally analytically (or \'etale locally) over $Y$, with transition functions unique only up to invertible scalars. Therefore choices need not satisfy the cocycle condition on triple overlaps but give a Brauer class
$$
\alpha\in H^2_{\mathrm{\mathaccent19 et}}(\cO^*_Y).
$$
The result is an $\alpha$-\emph{twisted} universal spinor sheaf
\beq{twis}
\mathcal U\in D\big(\cH\times_{(\P{2n+1})^*}\!Y,\,\alpha\big).
\eeq
Away from the codimension-3 locus of corank $\ge2$ quadrics, using \eqref{twis} as a Fourier-Mukai kernel gives a full and faithful embedding $D(Y,\alpha)\into D(\cH)$, making the twisted variety $(Y,\alpha)$ the HP dual of $(\P{2n+1},\cO(2))$. (It can even be extended locally over the locus of corank $2$ quadrics by taking a small resolution of $Y$ along its corank 2 locus \cite{Ad}.)

More globally one cannot quite produce an entirely geometric HP dual. Instead one works with sheaves of Clifford algebras over $(\P{2n+1})^*$. Generically these split as a direct sum of two matrix algebras, so they define a double cover $Y$ over which we get an Azumaya algebra; this is the correct HP dual \cite{Kuq}. Since modules over sheaves of matrix algebras are equivalent to modules over their centre we generically get modules over $\cO_Y$, but in high codimension we have no such geometric description. Instead one should think of the derived category of modules over the Azumaya algebra on the double cover $Y$ to be a ``noncommutative resolution" of the singularities of $Y$ (with its Brauer class $\alpha$).

For generic linear subsystems of dimension $\le2$ (or $\le3$ if we use a small resolution of $Y$) we can avoid the noncommutative locus and consider the HP dual to be a branched double cover. The upshot is that the derived category of an intersection of $r\le4$ quadrics and the derived category of the associated double cover of $\PP^{r-1}$ (with Brauer class and small resolution if $r=4$) are related by full and faithful embedding. Versions of this statement at various levels of generality appear in \cite{ABB,Ad,BO1,BO2,Ka,Kuq}.

For instance, taking $n=2$ and $r=3$ gives Mukai's derived equivalence \cite{Mu} between the K3 intersection of 3 quadrics $$S=Q_0\cap Q_1\cap Q_2$$ in $\P5$ and the (twisted) K3 double cover $M$ of the linear system $\P2=\PP\langle Q_0,Q_1,Q_2\rangle$ branched over a sextic. We emphasise again how simple HPD is to visualise here. Points of $M$ parameterise a choice of quadric through $S$ and spinor sheaf over it. Restricting this sheaf to $S$ makes $M$ into a moduli space of sheaves on $S$, and the (twisted) universal sheaf gives the derived equivalence.

\subsection[Example]{Bilinears}\label{biq}
A similar story has been found recently by Hori, Hosono-Takagi and Rennemo \cite{Ho, HT1, HT2, Re1}, replacing quadratic forms on $V$ by the associated symmetric bilinear forms on $V\times V$. Thus quadric hypersurfaces in $\PP(V)$ are replaced by (1,1) bilinear hypersurfaces in the orbifold
$$
\frac{\PP(V)\times\PP(V)}{\Z/2}\ =\ \Sym^2\PP(V).
$$
(Note this Deligne-Mumford stack is birational to $\Hilb^2\PP(V)$.) For the appropriate Lefschetz collection, the HP dual again admits a description in terms of sheaves of Clifford algebras over the full linear system $\PP(\Sym^2 V^*)$.\medskip

\noindent\textbf{dim$\;(V)$ even.} Here we get a picture somewhat like the even dimensional quadrics case, but with a non-rectangular Lefschetz collection. Over a big open set (and over small linear systems therein) the HP dual is a double cover of $\PP(\Sym^2 V^*)$ branched over the corank 1 locus. In contrast to the quadrics case, however, the Brauer class vanishes. \medskip

\noindent\textbf{dim$\;(V)$ odd.} It turns out now that the interesting part of the derived category of the generic (i.e. smooth) bilinear hypersurface is empty. Over the locus of corank 1 bilinear forms we get two exceptional orthogonal sheaves --- i.e. the derived category of 2 distinct points. These coalesce over the corank 2 locus.

So again over a big open set the HP dual is effectively a double cover --- this time of the corank 1 locus in $\PP(\Sym^2 V^*)$, branched over the corank 2 locus. (Thus --- generically --- we get a \emph{smooth} double cover of a \emph{singular} space, branched over its singular locus. In contrast, in the dim$\;V$ even case and the quadrics case, the double cover was generically over a smooth base with smooth branching locus.) \medskip

It would be interesting to study other invariants of these families of bilinears (intermediate Jacobians, Chow groups, motives, point counts over finite fields, etc) just as people have done for families of quadrics. \medskip

The above quadrics and bilinears are the first in a series examples. Let
$$
X_r(V):=\big\{[T]\in\PP(\Sym^2(V))\colon\rk(T)\le r\big\}\subset\PP(\Sym^2(V))
$$
be the locus of rank $\le r$ symmetric 2-tensors. For $r=1$ this is the Veronese embedding $\PP(V)\into\PP(\Sym^2V)$, of Example \ref{d=2}. For $r=2$ it is\footnote{It is the singular \emph{variety} $\Sym^2$, rather than the \emph{stack} $\Sym^2$ which Rennemo studies. We should think of the latter as a resolution of the former.} $\Sym^2\PP(V)\into\PP(\Sym^2V)$ of Example \ref{biq}. Set $n=\dim V$.

Hori and Knapp \cite{HK} and Kuznetsov \cite{ICM} conjecture that when $n-r$ is odd, the HP dual of $X_r(V)$ is a double cover of $X_{n-r+1}(V^*)$, branched over $X_{n-r}(V^*)$. (Thanks to Nick Addington and J\o rgen Rennemo for discussions about this.) As usual one has to use appropriate noncommutative resolutions (found recently in \cite{SvdB}) and possible Brauer classes on both sides.

\subsection[Example]{Pfaffian-Grassmannian}\label{PG}
Replacing the symmetric tensors in $\Sym^2V$ of the previous two examples by skew tensors in $\Lambda^2V$ gives the Pfaffian-Grassmannian duality of \cite{KuGr}. Let $V$ be a vector space of even dimension $\dim V=2n$. Its Pfaffian variety
$$
\Pf(V)=\big\{[\omega]\in\PP(\Lambda^2V^*)\colon\operatorname{corank}\omega\ge2\big\}\,\subset\,\PP(\Lambda^2V^*)
$$
is the degree $n$ hypersurface $\{\omega^n=0\big\}\subset\PP(\Lambda^2V^*)$ of degenerate 2-forms. It is the classical projective dual of the Grassmannian
$$
\Gr(2,V)=\big\{[P]\in\PP(\Lambda^2V)\colon\rk P=2\big\}\subset\PP(\Lambda^2V).
$$
It is also singular along the locus $\{\operatorname{corank}\omega>2\}$. Kuznetsov \cite{KuGr} finds natural (non-rectangular) Lefschetz collections and conjectures (and has proved in low dimensions) that there exists a noncommutative resolution of $\Pf(V)$ which makes it HP dual to $\Gr(2,V)$.

Thinking of a 2-form $\omega$ as a (skew) map $V\to V^*$ and $P$ as a plane $\subset V$, there is a natural correspondence
$$
\Gamma:=\Big\{(\omega,P)\in\Pf(V)\times\Gr(2,V)\colon\ker\omega\cap P\ne0\Big\}.
$$
In low dimensions at least, the Fourier-Mukai kernel which Kuznetsov conjectures sets up the HP duality is quasi-isomorphic to the ideal sheaf $\I_\Gamma$ of this correspondence.\medskip

\noindent\textbf{Generalised Pfaffians.}
Again the above examples sit inside a bigger series. Fix any vector space $V$ and set $n:=\lfloor\dim V/2\rfloor$. For any integer $0\le i\le n$ Kuznetsov conjectured that the generalised Pfaffian varieties
$$
\Pf_{2i}(V)=\big\{[\omega]\in\PP(\Lambda^2V^*)\colon\operatorname{rank}\omega\le2i\big\}\,\subset\,\PP(\Lambda^2V^*)
$$
and $\Pf_{2n-2i}(V^*)\subset\PP(\Lambda^2V)$ admit noncommutative resolutions which are HP dual. (Setting $\dim V=2n$ and $i=n-1$ gives the previous example.)

Hori \cite{Ho} proposed a physical duality between certain
non-abelian gauged linear sigma models which should realise this conjecture. Very recently Rennemo and Segal \cite{RS} have found a rigorous construction of Hori-dual B-brane categories for these models, thus proving Kuznetsov's conjecture in most cases.

\medskip\noindent\textbf{Borisov-C{\u{a}}ld{\u{a}}raru example.}
The $\dim V=7,\ i=2$ case, is proved in \cite{KuGr}. Applying HPD II to a 14-dimensional $L\subset\Lambda^2V$ whose 7-dimensional orthogonal $L^\perp\subset\Lambda^2V^*$ misses the singularities of $\Pf_4(V)$ gives the Borisov-C{\u{a}}ld{\u{a}}raru-Kuznetsov example of an equivalence of derived categories
$$
D\big(\!\Pf_4(\C^7)\cap\P6\big)\,\cong\,D\big(\!\Gr(2,\C^7)\cap\P{13}\big)
$$
between two non-birational Calabi-Yau 3-folds \cite{BC, KuGr}. This is reproved using sophisticated new techniques (variation of non-abelian GIT quotients, window subcategories and matrix factorisations) in \cite{HT, ADS}.\medskip

\noindent\textbf{Quintic threefold.}
The $\dim V=10,\ i=4$ case is not yet fully proved, but for appropriate $L\cong\C^{40}\subset\Lambda^2V$ it would imply a fully faithful embedding
$$
D\big(\!\Pf(\C^{10})\cap\P4)\big)\Into D\big(\!\Gr(2,10)\cap\P{39}\big).
$$
The left hand side is a quintic 3-fold,\footnote{Moreover Beauville \cite{Beau} has shown that the \emph{generic} quintic 3-fold has such a Pfaffian description.} the standard example of a Calabi-Yau 3-fold. The right hand side is a Fano 11-fold. This embedding is proved in \cite{ST} using methods from matrix factorisations and the variation of GIT quotients.

\medskip\noindent\textbf{Determinantal loci.} We have considered symmetric and skew symmetric matrices; one can also work with more general matrices and their determinantal varieties
$$
D_r(V,W)\,:=\,\big\{\Phi\in\PP(\Hom(V,W))\colon\rk\Phi\le r\big\}.
$$

In \cite{BBF}, it is proved that, after passing to appropriate resolutions or noncommutative resolutions,  
$$
D_r(V,W)\,\subset\,\PP(\Hom(V,W)) \quad\mathrm{and}\quad D_{n-r}(W,V)\,\subset\,\PP(\Hom(V,W)^*)
$$
become HP dual. Here $n=\min(\dim V,\dim W)$.

\medskip There are further beautiful examples of HPD in Kuznetsov's ICM survey \cite{ICM}. We have not touched on the close connection between HPD and categories of matrix factorisations of birational LG models, from which HPD can be deduced \cite{B+,Re2}. Many of the examples described here were first found in that context by Hori and his collaborators, and Hori has conjectured further powerful dualities in \cite{Ho}.

\subsection[Example]{Cubic fourfolds}\label{d=3}
Here we consider the HP dual of $(\P5,\cO(3))$ with its Lefschetz decomposition $D(\P5)=\Langle\cA,\cA(3)\Rangle,\ \cA=\langle\cO,\cO(1),\cO(2)\rangle$. This example nicely ties together parts of Examples \ref{d=2} and \ref{PG}.

The smooth hypersurfaces are now cubic fourfolds $H$, and $\cC_H$ is a ``noncommutative K3 category" --- it has Serre functor $[2]$ and the same Hochschild homology and cohomology as $D(\mathrm{K3})$. Thus the HP dual $\cC_\cH$ of $(\P5,\cO(3))$ is a ``noncommutative K3 fibration" over $\PP(\Sym^3\C^{6*})$.
 
For some \emph{special} cubic fourfolds $H$, the category $\cC_H$ is \emph{geometric}:
$$
\cC_H\,\cong\,D(S)
$$
for some (commutative!) algebraic K3 surface $S$. In \cite{KuC} Kuznetsov describes three families of examples.

\medskip \noindent\textbf{Pfaffian cubics.}
We begin with the Pfaffian-Grassmannian duality of Example \ref{PG} with $\dim V=6$ and $\dim L=9$. This gives
$$
D(S)\Into D(H),
$$
where $S=\Gr(2,6)\cap\P8$ is a K3 surface and $H=\Pf(\C^6)\cap\P5$ is a Pfaffian cubic fourfold. In \cite[Section 3]{KuC} Kuznetsov shows that the image of this embedding is precisely $\cC_H$.\medskip

\noindent\textbf{Nodal cubics.}
An attractive example using singular cubics considers those containing a single ODP $p\in H$. Pick a complimentary $\P4\subset\P5$ and consider it as the space of lines through $p$. Projecting from $p$ gives a birational map $H\dashrightarrow\P4$ since generic lines through $p$ intersect $H$ in only $3-2=1$ further point. The map blows up $p$ but blows down all those points on a line through $p$ which lies entirely inside $H$. The locus of such lines is easily seen to be a $(2,3)$ complete intersection in the $\P4$ --- i.e. a K3 surface $S$. The result is an isomorphism
\beq{Bll}
\Bl_pH\,\cong\,\Bl_S\P4.
\eeq
The obvious pull-up push-down functor
$$
D(S)\To D(H)
$$
can be mutated into an equivalence $D(S)\rt\sim\cC_H$ \cite[Section 5]{KuC}. (Because $H$ is singular the definition of $\cC_H$ has to be modified \cite{KuC}.)

\medskip
\noindent\textbf{Cubics containing a plane.}
The final example is given by cubic fourfolds $H$ containing a plane $P\subset H$. The linear system of
hyperplanes containing P defines a map $\Bl_P(H)\to\P2$ whose fibres are
quadric surfaces $Q$ since the intersection of two hyperplanes in $H$ is the reducible cubic surface $Q\cup P$.

As in Example \ref{d=2}, this quadric fibration over $\P2$ defines a K3 double cover $S\to\P2$, a Brauer class $\alpha$ and an embedding
\beq{Hura}
D(S,\alpha)\Into D(\Bl_P(H)).
\eeq
This can be projected into $D(H)$ and mutated into $\cC_H$ fully faithfully \cite[Section 4]{KuC}. Thus, when $\alpha=0$ (which is when the quadric fibration has a section) the category $\cC_H$ is geometric:
\beq{ratio}
D(S)\,\cong\,\cC_H.
\eeq

\medskip
\noindent\textbf{Rationality.} In the three families above where $\cC_H$ is geometric, the special cubic fourfolds are all \emph{rational}. For instance, in the third example, fibrewise stereographic projection from the section makes the quadric bundle birational to a Zariski-locally trivial $\P2$-bundle over $\P2$, which in turn is birational to $\P2\times\P2$ and so $\P4$.

In the second example rationality is especially explicit \eqref{Bll}, as is its connection to the geometric K3 category. Since, similarly, any birational map from a cubic fourfold to $\P4$ is expected to involve a blow up in a K3 surface, Kuznetsov conjectured that $\cC_H$ should be rational if and only if $\cC_H$ is equivalent to the derived category of some K3 surface $S$. (Note however that as yet, no cubic fourfold is proved to be irrational.)

Kuznetsov's results categorify earlier work explaining the remarkable connections between cubic fourfolds and K3 surfaces. In particular Hassett \cite{Ha1} defines a cubic fourfold $H$ to be \emph{special} if its Hodge structure contains an isometric copy of the primitive Hodge structure of a polarised $K3$ surface. He and Harris considered the question of whether this condition might be equivalent to the rationality of $H$. See \cite{Ha2} for a survey of this question.

\medskip
\noindent\textbf{Hassett vs Kuznetsov.} Hassett \cite{Ha1} shows that cubic fourfolds which are special in his Hodge theoretic sense are precisely those whose $H^{2,2}$ contains a primitive integral class $T$ satisfying a numerical condition (that the discriminant $d$ of the lattice $\langle h^2, T\rangle$ is not divisible by 4, 9 or any odd prime $p$ of the form $3n+2$). Such cubics form an irreducible Noether-Lefschetz divisor $C_d$ inside the 20 dimensional moduli space of cubic fourfolds.

When $\cC_H$ is geometric, equivalent to $D(S)$ for some algebraic K3 surface $S$, the induced map $D(S)\into D(H)$ induces a map $$H^*(S)\To H^*(H)$$ which makes $H$ special \cite{AT}. Thus it lies in one of the divisors $C_d$ above, and cubics which are special in the sense of Kuznetsov are also special in the sense of Hassett.

The converse is also expected to be true, and is proved \emph{generically} in \cite{AT}. More precisely, over a Zariski open (and so dense) subset of any nonempty special Noether-Lefschetz divisor $C_d$, the categories $\cC_H$ are all equivalent to the derived categories of K3 surfaces. This is proved by some deformation theory which is outside the scope of these notes, but the starting point of the deformation theory is pure HPD. Namely, to begin we need a cubic fourfold in $C_d$ which is Kuznetsov special. This is gotten by proving that $C_d$ contains cubics $H$ which contain a plane. Thus \eqref{ratio} gives us an equivalence
\beq{night}
\cC_H\,\cong\,D(S)
\eeq
(the Brauer class vanishes because of the existence of the special class $T$). Unfortunately this expresses $S$ as a moduli space of objects in $\cC_H$ which do \emph{not} deform  as $H$ moves through $C_d$ (as we saw in \eqref{ratio}, they are the pushforward to $H$ of spinor sheaves on the fibres of the quadric fibration $\Bl_PH\to\P2$). But from our special class $T\in H^{2,2}(H)$ (which \emph{does} deform throughout $C_d$) we now get an induced class in $H^*(S)$, and the classic results of Mukai then give a new K3 surface $M$ which is a fine moduli space of sheaves on $S$ in this class. The universal sheaf gives a derived equivalence
$$
D(S)\,\cong\,D(M)
$$
which we compose with \eqref{night} to give
\beq{2equiv}
\cC_H\,\cong\,D(M).
\eeq
It is $M$, and the derived equivalence \eqref{2equiv}, which we then deform over $C_d$. As $H$ deforms, we deform $M$ with it via the abstract Torelli theorem by insisting that its period point matches that of $H$ under the isomorphism of Hodge structures given to us by Hassett. As we do so we try to deform the Fourier-Mukai kernel of the equivalence \eqref{2equiv}. These deformations are governed by Hochschild cohomology, which for K3 categories is isomorphic to the Hodge structure of the K3 surface (with a modified grading). So deformations are governed by Hodge theory, and we chose our deformation of $M$ to make its Hodge structure line up with that of $H$. Thus the Fourier-Mukai kernel indeed deforms, and to all orders \cite{AT}.

Huybrechts \cite{Hu} has recently extended these results to twisted K3 surfaces, i.e. the question of which cubics $H$ satisfy the weaker condition $\cC_H\cong D(S,\alpha)$ for possibly nonzero $\alpha$. Again the starting point is Kuznetsov's HPD result \eqref{Hura}.

\medskip
\noindent\textbf{Pencils of cubics.} From a pencil of cubic fourfolds we can define two different Calabi-Yau 3-folds. HPD then makes them derived equivalent.

The first is the baselocus of the pencil --- i.e. the 3-dimensional Calabi-Yau intersection of two cubics in $\P5$. Secondly, consider the universal family $\cH\to\PP^1$ of cubic fourfolds $H$ in the pencil; passing to their K3 categories $\cC_H\subset D(H)$ gives a ``noncommutative K3 fibration" over $\P1$.

To make this latter object geometric, we can take the pencil to consist entirely of special cubics. Unfortunately this makes both $\cH$ and the baselocus singular, but in \cite{CT1} it is shown how to resolve this problem (crepantly) when the special cubics are either nodal or contain a plane. The result is derived equivalences between complete intersection Calabi-Yau 3-folds and K3-fibred Calabi-Yau 3-folds.

\subsection[Example]{Baseloci and blow ups} The initial data for HPD was a basepoint-free linear system, but an extension to linear systems with baselocus has been described recently \cite{CT2}.

Carocci and Turcinovic start with homologically dual varieties
$$
X\To\PP(V^*) \quad\mathrm{and}\quad Y\To\PP(V).
$$
(In fact, either could be noncommutative. So for instance a variety $X$ with a Lefschetz decomposition is all the data we need; we then take $\cC_{\cH_{\!X}}$ for $Y$.)

Fix a linear subsystem $W\subset V$. It defines a rational map $X\dashrightarrow\PP(W^*)$ which blows up its baselocus $X_{\PP(W^{\perp\!})}\subset X$ (the basechange of $X$ to $\PP(W^\perp)\subset\PP(V^*)$). Thus we get a regular map
$$
\Bl_{X_{\PP(W^{\perp\!})}}(X)\To\PP(W^*)
$$
and --- for the appropriate choice of Lefschetz collection \cite{CT2} --- its HP dual is $Y_{\PP(W)}\to\PP(W)$. This basechange of $Y\to\PP(V)$ to $\PP(W)$ can be geometric even when $Y$ itself is noncommutative; see examples in \cite{CT1,CT2}.

\bibliographystyle{halphanum}
\bibliography{references}

\medskip\noindent {\tt{richard.thomas@imperial.ac.uk}} \medskip

\noindent Department of Mathematics \\
\noindent Imperial College London\\
\noindent London SW7 2AZ \\
\noindent United Kingdom

\end{document}